\documentclass[11pt]{article}
\usepackage{pifont}

\usepackage{amsfonts}
\usepackage{latexsym}
\usepackage{amsmath}
\usepackage{amssymb}
\usepackage{color}

\newtheorem{theorem}{Theorem}[section]

\newtheorem{proposition}{Proposition}[section]

\newenvironment{proof}[1][Proof]{\noindent \textbf{#1.} }{\ \ \  $\Box$}

\newtheorem{lemma}{Lemma}[section]
\newtheorem{definition}{Definition}[section]
\newtheorem{remark}{Remark}[section]

\title{Zero-sum linear quadratic stochastic integral games and BSVIEs \thanks{This work is supported by National Natural
Science Foundation of China Grant 10771122, Natural Science
Foundation of Shandong Province of China Grant Y2006A08 and National
Basic Research Program of China (973 Program, No. 2007CB814900).} }

\date{May 28 2010}

\author{Tianxiao Wang and Yufeng Shi \thanks{E-mail: xiaotian2008001@gmail.com yfshi@sdu.edu.cn}\\ \small{School of
Mathematics, Shandong University, Jinan 250100, China}}

\begin{document}

\maketitle

\begin{abstract}
This paper formulates and studies a linear quadratic (LQ for short)
game problem governed by linear stochastic Volterra integral
equation. Sufficient and necessary condition of the existence of
saddle points for this problem are derived. As a consequence we
solve the problems left by Chen and Yong in \cite{CY}. Firstly, in
our framework, the term $GX^2(T)$ is allowed to be appear in the
cost functional and the coefficients are allowed to be random.
Secondly we study the unique solvability for certain coupled
forward-backward stochastic Volterra integral equations (FBSVIEs for
short) involved in this game problem. To characterize the condition
aforementioned explicitly, some other useful tools, such as backward
stochastic Fredholm-Volterra integral equations (BSFVIEs for short)
and stochastic Fredholm integral equations (FSVIEs for short) are
introduced. Some relations between them are investigated. As a
application, a linear quadratic stochastic differential game with
finite delay in the state variable and control variables is studied.

\par  $\textit{Keywords:}$ Stochastic integral games, open-loop
controls, saddle points, linear quadratic optimal control problem,
coupled forward-backward stochastic Volterra integral equations,
backward stochastic Fredholm-Volterra integral equation
\end{abstract}



\section{Introduction}
Differential game is a classical problem, there are several
frameworks of investigating it as far as the strategies are
concerned; see \cite{E}, \cite{H1}, \cite{H2}, \cite{I}, \cite{L}
for open-loop strategies of both deterministic and stochastic
differential games, and \cite{B}, \cite{F}, \cite{H2}, \cite{I},
\cite{L}, \cite{U}, \cite{V} for closed-loop strategy counterpart.
In addition, we would also like to mention the work of Fleming and
Souganidis \cite{FS}, who firstly gave a study on two-player
zero-sum stochastic differential games. Nonetheless there are few
literature to demonstrate some analytical research on the more
general dynamic setting, such as the system driven by a Volterra
equation. In this connection the only paper we know is You
\cite{Y3}, where the problem is spread out in the deterministic
setting. Such a lack of study is certainly not due to the
unimportance and non-interestingness of the problem, rather it is
because, we believe, that most of the effective techniques in the
conventional differential game are not analyzable and applicable
well in such setting. For example, as compared with differential
dynamic system, the time-consistency (or semi-group) property is
failure in the Volterra integral case, thus many good results along
this no longer hold.

In this paper, we will initiate a study on zero-sum linear quadratic
(LQ for short) stochastic integral game. A particularly important
notion for investigating the problem is given by Nash equilibria.
Loosely speaking, in the game problem, Player 1 wishes to minimize
the quadratic performance (4) which represents the cost and Player 2
wishes to maximize (4) which represents the payoff. Since both of
the two players are non-cooperative, they would like to seek their
admissible controls $\widehat{u}_1$ and $\widehat{u}_2$,
respectively, such that
\begin{equation}
J(\widehat{u}_1,u_2)\leq J(\widehat{u}_1,\widehat{u}_2)\leq J(u_1,\widehat{u}%
_2),
\end{equation}
for all the admissible controls $u_1$ and $u_2.$ The reason for (1)
holds is
that none of the players can improve his/her outcome $J(\widehat{u}_1,%
\widehat{u}_2)$ by deviating from $\widehat{u}_1$ or $\widehat{u}_2$
unilaterally. Thus both players will be satisfied with the controls $%
\widehat{u}_1$ and $\widehat{u}_2,$ respectively. In our framework, we refer to $(\widehat{u}%
_1,\widehat{u}_2)$ as an open-loop saddle point of the game over
$[0,T],$ and we consider only the open loop strategies in the
following part. Additionally, we point out that in general, an
open-loop saddle point (if it exists) is not necessarily unique.

Before demonstrating the model we will study in this paper, we would
like to summarize how our work relates to the literature on which it
builds. One of the main results is establishing the relation between
the LQ stochastic integral game aforementioned and backward
stochastic Volterra integral equations (BSVIEs for short), which is
based on the results in \cite{CY}, nonetheless in a much more
general setting. First the two-player nature of the game problem
demands more delicate manipulations of all the involved Hilbert
operators. Secondly we solve a problem left by Chen and Yong in
\cite{CY} since here the term $GX^{2}(T)$ is allowed to appear in
the cost functional (4). As a consequence, our paper naturally
relates to the BSVIEs theoretic literature. The unique solvability
of BSVIEs was firstly studied by Lin \cite{L}, see also \cite{WZ}
for later related research. As to the following general form,
\begin{eqnarray}
Y(t)=\psi
(t)+\int_t^Tg(t,s,Y(s),Z(t,s),Z(s,t))ds-\int_t^TZ(t,s)dW(s),
\end{eqnarray}
we should mention the contribution of Yong (\cite{Y1}, \cite{Y2}),
who firstly introduced the notion of adapted M-solution in \cite{Y2}
and established a Pontryagin type maximum principle for optimal
control of stochastic Volterra integral equations with the help of
M-solution. Moreover, as shown by Yong in \cite{Y2'}, a class of
continuous-time dynamic convex and coherent risk measures can be
derived via certain BSVIEs. Along this we also refer the reader to
\cite{R}, \cite{WS} for some other studies on this topic.

In a deterministic setting, the game problem for an input-output
system governed by Volterra integral equation of the form
\[
X(t)=\varphi (t)+\int_0^t[B_1(t,s)u_1(s)+C_1(t,s)u_2(s)]ds
\]
with respect to a quadratic performance functional of the form
\[
J(u_1(\cdot ),u_2(\cdot
))=E\int_0^T[QX^2(t)+R_1u_1^2(t)+R_2u_2^2(t)]dt+EGX^2(T),
\]
was studied in \cite{Y3}. In this context, $X,u_1$ and $u_{2%
} $ are deterministic functions. With this interpretation, our
problem here can be seen as a natural extension of the work in
\cite{Y3}. To our knowledge, the paper might be the first one for
the stochastic quadratic integral game. As to the causal feedback
optimal control problem for deterministic Volterra integral
equation, we would like to mention the works in \cite{Y3} and
\cite{PY}. A new method called $projection$ $causality$ to this
problem was introduced. So when the system is stochastic Volterra
integral equation, how to provide a causal feedback implementation
of the optimal strategies is still a problem and we hope to study it
in the future.

In the traditional stochastic differential game, coupled
forward-backward stochastic differential equations (FBSDEs for
short) play an important role in the existence of the open-loop
saddle points, see for example \cite{MY1}, \cite{Y0} and the
reference cited therein. As to the solvability of coupled FBSDEs,
there have been burgeoning research interest in it, see \cite{A},
\cite{HP}, \cite{MPY}, whereas almost all the methods depend heavily
on It\^{o} formula or the time-consistent (or semi-group) property
of differential equation. In our framework, we will obtain the
existence of an open loop saddle point of the quadratic integral
game, which is equivalent to the solvability of certain FBSVIE plus
the convexity and concavity of the cost functional below. Thus the
solvability of coupled forward-backward stochastic Volterra integral
equations (FBSVIEs for short) should also play an important role for
the stochastic integral game we are tackling. However, the
solvability for FBSVIEs is more challenging as compared with the
situations for FBSDEs aforementioned. On the one hand, many
conventional and convenient approaches or conditions, such as the
four-steps method in \cite{MPY}, the monotonicity condition in
\cite{HP}, especially the most important It\^{o} formula, all are
absent in this case. On the other hand, become of the lack of
time-consistent property for BSVIEs (or FBSVIEs), we can not use the
induction directly as differential equation and more complicated
things should be involved, see the existence and uniqueness of
M-solution of BSVIEs in \cite{Y2} for detailed accounts. Worse
still, the coupling of there two factors greatly amplifies the
difficulty of the problem.

In this paper, given assumptions, we will establish the existence
and uniqueness of M-solution for coupled FBSVIEs (31),
which will be involved in our game problem. As mentioned earlier
most of effective techniques in tackling the problem for
differential equation become failure, therefore, we have to carry
out investigation from some other basic and original views. By
assuming that $\beta$ is a constant, we can introduce a new
equivalent norm for M-solutions of FBSVIEs as follows
\begin{eqnarray*}
&&\left\| (x(\cdot ),y(\cdot ),z(\cdot ,\cdot ))\right\| _{\mathcal{H}%
^2[0,T]\times L^2_{\mathcal{F}}[0,T]}^2 \\
&=&E\left[ \int_0^Te^{-\beta s}|x(s)|^2ds+\int_0^Te^{\beta
s}|y(s)|^2ds+\int_0^Te^{\beta t}\int_0^T|z(t,s)|^2dsdt\right],
\end{eqnarray*}
thereby study the unique existence by means of fixed point theorem.
This is a common trick employed in the conventional BSDEs case,
which also enables us to get around the inapplicability of It\^{o}
formula in the current setting. It is also worthy to claim that we
do not need more assumptions, such as the monotonicity condition in
\cite{HP}, except the Lipschitz condition. Thus this can be seen as
another contribution of this paper.

Notice that as to general form of coupled FBSVIEs, see (32) below,
there is hitherto no well technology to deal with, and it is an
object of endeavor for us in the future. One substantial difficulty,
we believe, was caused by the appearance of $X(T)$ in the two
backward equations of (32). However, in certain special case, we can
transform the unique solvability of FBSVIE (32) into the solvability
of some kind of backward stochastic Fredholm-Volterra integral
equation (BSFVIE for short), allowing the appearance of $X(T).$ More
importantly, under some assumptions, the aforementioned BSFVIE
becomes a forward stochastic Fredholm-Volterra integral equation
(SFVIE for short), and this helps us to characterize the Nash
equilibrium strategy more explicitly. We refer the readers to
\cite{CW} and \cite{V} for details on the solvability of SFVIEs. As
to the case for BSFVIE, the problem is much more complicated and we
hope to study it in the future. At last we will illustrate the
application of the obtained results to the stochastic quadratic
differential game with delay.

The reminder of this paper is organized as follows. In the next
section, the game problem will be formulated and some preliminary
results will also be stated. In Section 3 we study the LQ integral
games in Hilbert space and obtain one necessary and sufficient
condition of existence of saddle point. In Section 4 we will make
use of BSVIEs (or FBSVIEs) to characterize the result derived in
Section 3 more explicitly. In Section 5 we give another sufficient
condition with the help of the solvability of M-solution for coupled
FBSVIE (26). Some further considerations, such as the relationship
between coupled FBSVIEs, BSFVIEs and SFVIEs are investigated. At
last we give some results of stochastic differential game with delay
and obtain one explicit expression of the saddle point, which is
consistent with the result in \cite{Y3}.

\section{Problem formulation and preliminary}
Let $(W_t)_{t\in [0,T]}$ be a scalar-valued Wiener process defined
on a probability space $(\Omega ,\mathcal{F},P)$ and
$(\mathcal{F}_t)_{t\in [0,T]} $ denotes the natural filtration of
$(W_t),$ such that $\mathcal{F}_0$ contains all $P$-null sets of
$\mathcal{F}.$ Our assumption that $W(\cdot )$ is scalar-valued is
for the sake of simplicity and no essential difficulties are
encountered when extending our analysis to the case of vector-valued
Brownian motion.

Suppose the dynamic of a stochastic system is described by a
controlled linear stochastic Volterra integral equation (SVIE for
short),
\begin{eqnarray}
X(t) &=&\varphi (t)+\int_0^t[A_1(t,s)X(s)+B_1(t,s)u_1(s)+C_1(t,s)u_2(s)]ds  \nonumber \\
&&+\int_0^t[A_2(t,s)X(s)+B_2(t,s)u_1(s)+C_2(t,s)u_2(s)]dW(s),
\end{eqnarray}
where $u_1$ and $u_2$ are adapted and stand for, respectively the
intervention functions of two agents Play 1 and Play 2 on the
dynamic system. $X$ is the state process and $u_1$ and $u_2$ are
control processes taken by two players. To avoid undue technicality,
we assume both the state process and control process are
scalar-valued. We define the cost functional associated with (3) for
the players as follows:
\[
J(u_1(\cdot ),u_2(\cdot
))=E\int_0^Tf(t,X(t),u_1(t),u_2(t))dt+EGX^2(T),
\]
where
\begin{eqnarray*}
&&f(t,X(t),u_1(t),u_2(t)) \\
&=&Q(t)X^2(t)+2S_1(t)X(t)u_1(t)+2S_2(t)X(t)u_2(t)+R_{11}(t)u_1^2(t)\\
&&+R_{12}(t)u_1(t)u_2(t)+R_{21}(t)u_1(t)u_2(t)+R_{22}(t)u_2^2(t)
\\
&=&\left\langle \left(
\begin{array}{ccc}
Q(t) & S_1(t) & S_2(t) \\
S_1(t) & R_{11}(t) & R_{12}(t) \\
S_2(t) & R_{21}(t) & R_{22}(t)
\end{array}
\right) \left(
\begin{array}{c}
X(t)\\
u_1(t) \\
u_2(t)
\end{array}
\right) ,\left(
\begin{array}{c}
X(t)\\
u_1(t) \\
u_2(t)
\end{array}
\right) \right\rangle _2.
\end{eqnarray*}
Note that $\left\langle\cdot,\cdot \right\rangle_2$ is defined blew.
In what follows, we will denote
\[
S(t)=\left(
\begin{array}{c}
S_1(t) \\
S_2(t)
\end{array}
\right) , \quad R(t)=\left[
\begin{array}{cc}
R_{11}(t) & R_{12}(t) \\
R_{21}(t) & R_{22}(t)
\end{array}
\right],\quad u(t)=\left(
\begin{array}{c}
u_1(t) \\
u_2(t)
\end{array}
\right) ,
\]
and
\begin{eqnarray}
&&J(u_1(\cdot ),u_2(\cdot ))\nonumber \\
&=&E\int_0^T[Q(t)X^2(t)+2X(t)S(t)\cdot u(t)+R(t)u(t)\cdot u(t)]dt+EGX^2(T)\nonumber \\
&=&\left\langle QX,X\right\rangle _2+2\left\langle XS,u\right\rangle
_2+\left\langle Ru,u\right\rangle _2+E\left\langle
GX(T),X(T)\right\rangle_1.
\end{eqnarray}
Throughout this paper, we assume that $Q$, $R_{ij}$ and $S_{i}$
($i,j=1,2$) are bounded adapted processes and $G$ is a bounded
random variable. Note that in the above model, the controls are
allowed to appear in both the drift and diffusion of the state
equation, the weighting matrices in the payoff/cost functional are
not assumed to be definite/non-singular, and the cross-terms between
two controls are allowed to appear, we refer this problem as a
so-called zero-sum linear quadratic stochastic integral game.

Next we will give some notations. We denote $\Delta^c=\{t,s)\in[0,T]^2; t\leq s\}$ and
$\Delta=[0,T]^2\ \Delta^c.$ Let $L^2(\Omega \times [0,T])$ be the set of the processes $%
X:[0,T]\times \Omega \rightarrow R$ which is
$\mathcal{B}([0,T])\times
\mathcal{F}_T$-measurable satisfying $E\int_0^T|X(t)|^2dt<\infty .$ $%
L^2(\Omega )$ is set of random variable $\xi :\Omega \rightarrow R$
which is $\mathcal{F}_T$-measurable satisfying $E|\xi |^2<\infty $,
and we denote its inner product by $\left\langle \cdot ,\cdot
\right\rangle _1.$ $\forall
R,S\in [0,T],$ $L_{\mathcal{F}}^2[R,S]$ is the set of all adapted processes $%
X:[R,S]\times \Omega \rightarrow R$ such that
$E\int_R^S|X(t)|^2dt<\infty $, and we denote its inner product by
$\left\langle \cdot ,\cdot \right\rangle
_2.$ $L^2(R,S;L_{\mathcal{F}}^2[R,S])$ be the set of all process $%
Z:[R,S]^2\times \Omega \rightarrow R$ such that for almost all $t\in
[R,S],$
$Z(t,\cdot )$ is $\mathcal{F}$-adapted satisfying $E\int_R^S%
\int_R^S|Z(t,s)|^2dsdt<\infty .$ We denote $\mathcal{H}^2[R,S]=L_{\mathcal{F}%
}^2[R,S]\times L^2(R,S;L_{\mathcal{F}}^2[R,S]).$ $L^{\infty} [0,T]$
is set of deterministic function $X:[0,T]\times\Omega\rightarrow R$
such that $\sup\limits_{t\in [0,T]}|X(t)|<\infty .$
$L^2(0,T;L^\infty [0,T])$ is set of deterministic function
$X:[0,T]^2\rightarrow R$ such that for almost $t\in [0,T],$
$\sup\limits _{s\in [0,T]}|X(t,s)|<\infty .$ $L^2(0,T;L^2[0,T])$ is
set of deterministic
function $X:[0,T]^2\rightarrow R$ such that for almost $t\in [0,T],$ $%
\int_0^T\int_0^T|X(t,s)|^2ds<\infty .$ As to
$L^{\infty}_{\Bbb{F}}[0,T]$, $L^\infty(0,T;L^2_{\Bbb{F}}[0,T])$ and
$L^\infty(0,T;L^\infty_{\Bbb{F}}[0,T])$, we can define them in a
similar manner.

The notion of M-solutions of BSVIEs can be expressed as,
\begin{definition}
Let $S\in [0,T]$. A pair of $(Y(\cdot ),Z(\cdot ,\cdot ))\in \mathcal{H}%
^2[S,T]$ is called an adapted M-solution of BSVIE (2) on $[S,T]$ if
(2) holds in the usual It\^o's sense for almost all $t\in [S,T]$ and
in addition, the following holds:
\[
Y(t)=E^{\mathcal{F}_S}Y(t)+\int_S^tZ(t,s)dW(s),\quad t\in [S,T].
\]
\end{definition}
In \cite{Y2}, the author gave the definition of M-solution of BSVIE in $%
\mathcal{H}^2[0,T]$ and proved the following proposition,
\begin{proposition}
Let $g:\Delta ^c\times R\times R\times R\times \Omega \rightarrow R$
be $\mathcal{B}(\Delta ^c\times R\times R
\times R)\otimes \mathcal{F}_T$-measurable such that $%
s\rightarrow g(t,s,y,z,\zeta )$ is $\Bbb{F}$-progressively
measurable for all $(t,y,z,\zeta )\in [0,T]\times R\times R\times
R$, moreover, $g$ satisfies the Lipschitz conditions, $\forall y,$
$\overline{y}\in R,$ $z,$ $\overline{z},$ $\zeta ,$ $\overline{\zeta
}\in R,$
\begin{eqnarray*}
\ \ |g(t,s,y,z,\zeta
)-g(t,s,\overline{y},\overline{z},\overline{\zeta })| \leq
L(t,s)(|y-\overline{y}|+|z-\overline{z}|+|\zeta -\overline{\zeta
}|),
\end{eqnarray*}
where $(t,s)\in \Delta ^c,$ $\Delta ^c=\left\{ (t,s)\in [0,T]^2\mid
t\leq s\right\} ,$ $L(t,s)$ is a determined non-negative function
satisfying $\sup\limits _{t\in
[0,T]}\displaystyle\int_t^TL^{2+\epsilon }(t,s)ds<\infty ,$ for some
$\epsilon >0.$ Then (2) admits a unique M-solution in
$\mathcal{H}^2[0,T].$
\end{proposition}

\section{Stochastic LQ integral games in Hilbert spaces}
In this section the linear quadratic stochastic integral games
problem is formulated in Hilbert space. It is important to recognize
that the classical LQ stochastic differential games and LQ optimal
control problem for FSVIEs can also be treated similarly in infinite
dimensional space, see \cite{MY1} and \cite{CY}. We incorporate some
useful techniques in Chen and Yong \cite{CY}, but investigate the
problem in a more general framework. On the one hand, the
coefficients, in both state equation and cost functional, are
allowed to be random, moreover, the form of cost functional is
general, especially allowing the appearance of the term $GX^2(T)$.
On the other hand, the nature of game problem also demand more
delicate analysis of the operators involved. To start with, we need
to make some preliminary.

Let $\mathcal{H}$ be a Hilbert space and $\Theta :\mathcal{D}(\Theta
)\subseteq \mathcal{H}\rightarrow \mathcal{H}$ be a self-adjoint
operator, i.e., it is densely defined and closed but not necessarily
bounded. We denote $\mathcal{R}(\Theta )$ and $\mathcal{N}(\Theta )$
to be the range and
kernel of $\Theta ,$ respectively. Since $\Theta $ is self-adjoint, $%
\mathcal{N}(\Theta )^{\perp }=\overline{\mathcal{R}(\Theta )}$ and we have $%
\Theta \left( \mathcal{D}(\Theta
)\bigcap\overline{\mathcal{R}(\Theta )}\right)
\subseteq \mathcal{R}(\Theta ).$ Thus under the decomposition $\mathcal{H}=%
\mathcal{N}(\Theta )\oplus \overline{\mathcal{R}(\Theta )},$ we have
the following representation for $\Theta :$ $\Theta =\left(
\begin{array}{cc}
0 & 0 \\
0 & \widehat{\Theta }
\end{array}
\right) ,$ where $\widehat{\Theta }:\mathcal{D}(\Theta )\bigcap\overline{\mathcal{%
R}(\Theta )}\subseteq \overline{\mathcal{R}(\Theta )}\rightarrow \overline{%
\mathcal{R}(\Theta )}$ is self-adjoint (again, it is densely defined
and
closed, but not necessarily bounded, on the Hilbert space $\overline{%
\mathcal{R}(\Theta )}).$ Now we define the pseudo-inverse $\Theta
^{\dagger } $ by the following: $\Theta ^{\dagger }=\left(
\begin{array}{cc}
0 & 0 \\
0 & \widehat{\Theta }^{-1}
\end{array}
\right) ,$ with domain
\[
\mathcal{D}(\Theta ^{\dagger })=\mathcal{N}(\Theta
)+\mathcal{R}(\Theta )\equiv \{u_0+u_1\mid u_0\in \mathcal{N}(\Theta
),u_1\in \mathcal{R}(\Theta )\}\supseteq \mathcal{R}(\Theta ).
\]
Let $\mathcal{H}=\mathcal{H}_1\times \mathcal{H}_2$ with
$\mathcal{H}_1$ and $\mathcal{H}_2$ being two Hilbert spaces, and we
consider a quadratic functional on $\mathcal{H}:$ for any
$u=(u_1,u_2),$ $v\in \mathcal{H},$
\begin{eqnarray*}
J(u) &\equiv &J(u_1,u_2)=\left\langle \Theta u,u\right\rangle
+2\left\langle
v,u\right\rangle \\
&\equiv &\left\langle \left(
\begin{array}{cc}
\Theta _{11} & \Theta _{12} \\
\Theta _{21} & \Theta _{22}
\end{array}
\right) \left(
\begin{array}{c}
u_1 \\
u_2
\end{array}
\right) ,\left(
\begin{array}{c}
u_1 \\
u_2
\end{array}
\right) \right\rangle \\
&&+2\left\langle \left(
\begin{array}{c}
v_1 \\
v_2
\end{array}
\right) ,\left(
\begin{array}{c}
u_1 \\
u_2
\end{array}
\right) \right\rangle .
\end{eqnarray*}
Here $\Theta _{ij}:\mathcal{H}_j\rightarrow \mathcal{H}_i$ ($i,j=1,2$) is bounded operator, $%
\Theta \equiv \left(
\begin{array}{cc}
\Theta _{11} & \Theta _{12} \\
\Theta _{21} & \Theta _{22}
\end{array}
\right) $ is self-adjoint. We have the proposition as follows:
\begin{proposition}
There exists a saddle point $(\widehat{u}_1,\widehat{u}_2)\in \mathcal{H}%
_1\times \mathcal{H}_2$ for $(u_1,u_2)\mapsto J(u_1,u_2),$ that is,
\[
J(\widehat{u}_1,u_2)\leq J(\widehat{u}_1,\widehat{u}_2)\leq J(u_1,\widehat{u}%
_2),\forall (u_1,u_2)\in \mathcal{H}_1\times \mathcal{H}_2,
\]
if and only if $v\in \mathcal{R}(\Theta )$ and the following are true: $%
\Theta _{11}\geq 0$ (it means that $\forall u_1\in \mathcal{H}%
_1,\left\langle \Theta _{11}u_1,u_1\right\rangle _{\mathcal{H}_1}\geq 0)$, $%
\Theta _{22}\leq 0$ (it means that $\forall u_2\in \mathcal{H}%
_2,\left\langle \Theta _{22}u_2,u_2\right\rangle
_{\mathcal{H}_2}\leq 0).$
In the above case, each saddle point $\widehat{u}=(\widehat{u}_1,\widehat{u}%
_2)\in \mathcal{H}_1\times \mathcal{H}_2$ is a solution of the equation: $%
\Theta \widehat{u}+v=0,$ and it admits a representation: $\widehat{u}%
=-\Theta ^{\dagger }v+(I-\Theta ^{\dagger }\Theta )\widetilde{v},$ for some $%
\widetilde{v}\in \mathcal{H}.$ Moreover, $\widehat{u}$ is unique if
and only if $\mathcal{N}(\Theta )=\{0\}.$
\end{proposition}
\begin{proof}We refer the reader to see the proof in \cite{MY1} or \cite{CY}.
\end{proof}

The above argument indicates that we could discuss the quadratic
integral game by using certain Hilbert operators. Before going
further, we need the following standing assumptions which is in
force in the rest of the paper.

(H1) $A_1\in L^\infty(0,T;L^2_{\Bbb{F}}[0,T])$, $A_2(\cdot ,\cdot
)\in L^{\infty}(0,T;L^{\infty}_{\Bbb{F}}[0,T]),$ $B_i(t,s)$ and
$C_i(t,s)$ $(i=1,2)$ also satisfy the similar assumption.

For any $(X,u_1,u_2)\in L_{\mathcal{F}}^2[0,T]\times L_{\mathcal{F}%
}^2[0,T]\times L_{\mathcal{F}}^2[0,T],$ we can define the operators $\mathcal{A},$ $\mathcal{B%
}_1,$ $\mathcal{C}_1$ from $L_{\mathcal{F}}^2[0,T]$ to itself as
follows:
\begin{eqnarray*}
(\mathcal{A}X)(t) &=&\int_0^tA_1(t,s)X(s)ds+\int_0^tA_2(t,s)X(s)dW(s), \\
(\mathcal{B}_1u_1)(t)
&=&\int_0^tB_1(t,s)u_1(s)ds+\int_0^tB_2(t,s)u_1(s)dW(s), \\
(\mathcal{C}_1u_2)(t)
&=&\int_0^tC_1(t,s)u_2(s)ds+\int_0^tC_2(t,s)u_2(s)dW(s),
\end{eqnarray*}
thus we have
\[
X(t)=\varphi(t)+(\mathcal{A}X)(t)+(\mathcal{B}_1u_1)(t)+(\mathcal{C}_1u_2)(t).
\]
The following lemma character the well property of the operators
defined above.
\begin{lemma}
Let (H1) hold, then the operators $\mathcal{A},$ $\mathcal{B}_1$ and $%
\mathcal{C}_1$ are bounded operators and $\mathcal{A}$ is
quasi-nilpotent,
i.e., $\overline{\lim\limits_{k\rightarrow \infty }}\left\| \mathcal{A}%
^k\right\| ^{\frac 1k}=0.$ Consequently, $(I-\mathcal{A)}^{-1}:L_{\mathcal{F}%
}^2[0,T]\rightarrow L_{\mathcal{F}}^2[0,T]$ is bounded, hence, for any $%
\varphi (\cdot )\in L_{\mathcal{F}}^2[0,T]$ and $u_1,u_2\in L_{\mathcal{F}%
}^2[0,T],$ (3) admits a unique solution $X=(I-\mathcal{A)}^{-1}(\varphi +%
\mathcal{B}_1u_1+\mathcal{C}_1u_2).$
\end{lemma}
\begin{proof}The proof is essentially resembles the one in \cite{CY} and we omit it
here.
\end{proof}

Due to the appearance of $GX^2(T),$ some other operators are also
required to tackle it. We denote
\begin{eqnarray*}
\Delta_TX &=&\int_0^TA_1(T,s)X(s)ds+\int_0^TA_2(T,s)X(s)dW(s), \\
\Lambda_Tu_1
&=&\int_0^TB_1(T,s)u_1(s)ds+\int_0^TB_2(T,s)u_1(s)dW(s), \\
\Pi_Tu_2 &=&\int_0^TC_1(T,s)u_2(s)ds+\int_0^TC_2(T,s)u_2(s)dW(s),
\end{eqnarray*}
hence
\[
X(T)=\Delta_TX+\Lambda_Tu_1+\Pi_Tu_2+\varphi (T).
\]
Obviously $\Delta_T,$ $\Lambda_T$ and $\Pi%
_T $ are bounded operators from $L_{\mathcal{F}}^2[0,T]$ to
$L^2(\Omega ).$

In what follows, we make some conventions as,
\begin{eqnarray*}
(\mathcal{U}u)(t)=(\mathcal{B}_1u_1)(t)+(\mathcal{C}_1u_2)(t),\quad
\Gamma_Tu=\Lambda_Tu_1+\Pi_Tu_2,
\end{eqnarray*}
therefore,
\begin{eqnarray}
\mathcal{U}u=(\mathcal{B}_1,\mathcal{C}_1)\left(
\begin{array}{c}
u_1 \\
u_2
\end{array}
\right) =(\mathcal{B}_1,\mathcal{C}_1)u, \\
\Gamma_Tu=(%
\Lambda_T,\Pi_T)\left(
\begin{array}{c}
u_1 \\
u_2
\end{array}
\right) =(\Lambda_T,\Pi_T)u.
\end{eqnarray}
We define the operators $\mathcal{Q},$ $\mathcal{S}$ and
$\mathcal{R}$ as follows: for $i,j=1,2,$
\begin{eqnarray*}
&&\left\langle \mathcal{Q}X,X\right\rangle
_2=E\int_0^TQ(t)X^2(t)dt,\quad \left\langle
\mathcal{S}X,u\right\rangle
_2=E\int_0^TS(t)X(t)\cdot u(t)dt, \\
&&\left\langle \mathcal{R}u,u\right\rangle _2=E\int_0^TR(t)u(t)\cdot
u(t)dt,\quad \left\langle \mathcal{S}_iX,u_i\right\rangle
_2=E\int_0^TS_i(t)X(t)u_i(t)dt, \\
&&\left\langle \mathcal{R}_{i,j}u_i,u_j\right\rangle
_2=E\int_0^TR_{i,j}(t)u_i(t)\cdot u_j(t)dt,
\end{eqnarray*}
consequently,
\begin{eqnarray}
\mathcal{S=}\left(
\begin{array}{c}
\mathcal{S}_1 \\
\mathcal{S}_2
\end{array}
\right), \quad \mathcal{R=}\left(
\begin{array}{cc}
\mathcal{R}_{11} & \mathcal{R}_{12} \\
\mathcal{R}_{21} & \mathcal{R}_{22}
\end{array}
\right),
\end{eqnarray}
and (4) can be rewritten as
\begin{eqnarray}
J(u)=\left\langle \mathcal{Q}X,X\right\rangle _2+2 \left\langle
\mathcal{S}X,u\right\rangle _2+\left\langle
\mathcal{R}u,u\right\rangle _2+\left\langle G X(T),X(T)\right\rangle
_1.
\end{eqnarray}
Now we turn to deal with $\left\langle G X(T),X(T)\right\rangle _1$
by means of the operators defined previously.
\begin{eqnarray}
&&\left\langle G X(T),X(T)\right\rangle _1 \nonumber\\
&=&\left\langle G(\Delta_TX+\Gamma_Tu+\varphi (T)),%
\Delta_TX+\Gamma_Tu+\varphi (T)\right\rangle _1  \nonumber\\
&=&\left\langle \Delta_T^{*}G\Delta_TX,X\right\rangle
_2+2\left\langle \Gamma_T^{*}G\Delta_TX,u\right\rangle
_2+\left\langle \Gamma_T^{*}G\Gamma_Tu,u\right\rangle
_2 \nonumber \\
 &&+2\left\langle X,\Delta_T^{*}G\varphi (T)\right\rangle
_2+2\left\langle u,\Gamma_T^{*}G\varphi (T)\right\rangle
_2+\left\langle \varphi (T),\varphi (T)\right\rangle _1,
\end{eqnarray}
where $\forall \eta \in L^2(\Omega ),$ $X,$ $u\in L_{\mathcal{F}}^2[0,T],$%
\[
\left\langle \Delta_TX,\eta \right\rangle _1=\left\langle X,%
\Delta_T^{*}\eta \right\rangle _2, \quad \left\langle \Gamma%
_Tu,\eta \right\rangle _1=\left\langle u,\Gamma_T^{*}\eta
\right\rangle _2.
\]
If we denote
\begin{eqnarray}
\mathcal{Q}^{\prime }=\mathcal{Q}+\Delta_T^{*}G\Delta%
_T,\quad \mathcal{S}^{\prime }=\mathcal{S}+\Gamma_T^{*}G%
\Delta_T,\quad \mathcal{R}^{\prime }=\mathcal{R}+\Gamma%
_T^{*}G\Gamma_T,
\end{eqnarray}
then we can obtain the following expressions after substituting (9)
into (8),
\begin{eqnarray*}
J(u)&=&\left\langle \mathcal{Q}^{\prime }X,X\right\rangle
_2+2\left\langle \mathcal{S}^{\prime }X,u\right\rangle
_2+\left\langle \mathcal{R}^{\prime
}u,u\right\rangle _2 \\
 &&+2\left\langle X,\Delta_T^{*}G\varphi
(T)\right\rangle _2+2\left\langle u,\Gamma_T^{*}G\varphi
(T)\right\rangle
_2+\left\langle \varphi (T),\varphi (T)\right\rangle _1 \\
&=&\left\langle \left(
\begin{array}{cc}
\mathcal{Q}^{\prime } & \mathcal{S}^{\prime *} \\
\mathcal{S}^{\prime } & \mathcal{R}^{\prime }
\end{array}
\right) \left(
\begin{array}{c}
(I-\mathcal{A})^{-1}(\varphi +\mathcal{U}u) \\
u
\end{array}
\right) , \left(
\begin{array}{c}
(I-\mathcal{A})^{-1}(\varphi +\mathcal{U}u) \\
u
\end{array}
\right)
\right\rangle _2 \\
&&+2\left\langle (I-\mathcal{A})^{-1}(\varphi +\mathcal{U}u),%
\Delta_T^{*}G\varphi (T)\right\rangle _2+2\left\langle u,\Gamma%
_T^{*}G\varphi (T)\right\rangle _2+\left\langle \varphi (T),\varphi
(T)\right\rangle _1 \\
&=&\left\langle \left(
\begin{array}{cc}
\mathcal{Q}^{\prime } & \mathcal{S}^{^{\prime }*T} \\
\mathcal{S}^{\prime } & \mathcal{R}^{\prime }
\end{array}
\right) \left(
\begin{array}{cc}
(I-\mathcal{A})^{-1} & (I-\mathcal{A})^{-1}\mathcal{U} \\
0 & I
\end{array}
\right) \left(
\begin{array}{c}
\varphi  \\
u
\end{array}
\right) ,\right.  \\
&& \left. \left(
\begin{array}{cc}
(I-\mathcal{A})^{-1} & (I-\mathcal{A})^{-1}\mathcal{U} \\
0 & I
\end{array}
\right) \left(
\begin{array}{c}
\varphi  \\
u
\end{array}
\right) \right\rangle _2 \\
&&+2\left\langle u,\mathcal{U}^{*}(I-\mathcal{A}^{*})^{-1}%
\Delta_T^{*}G\varphi (T)+\Gamma_T^{*}G\varphi (T)\right\rangle
_2 \\
&& +2\left\langle (I-\mathcal{A})^{-1}\varphi ,\Delta%
_T^{*}G\varphi (T)\right\rangle +\left\langle \varphi (T),\varphi
(T)\right\rangle _1 \\
&=&\left\langle \Theta u,u\right\rangle _2+\left\langle %
\Theta_1\varphi ,u\right\rangle _2+\left\langle \Theta_2\varphi
,\varphi \right\rangle _2 \\
&&+2\left\langle (I-\mathcal{A})^{-1}\varphi ,\Delta%
_T^{*}G\varphi (T)\right\rangle +\left\langle \varphi (T),\varphi
(T)\right\rangle _1,
\end{eqnarray*}
where
\begin{eqnarray}
&&\Theta=(\mathcal{U}^{*T}(I-\mathcal{A}^{*})^{-1}\mathcal{Q}%
^{\prime }\mathcal{+S}^{\prime })(I-\mathcal{A})^{-1}\mathcal{U+U}^{*T}(I-%
\mathcal{A}^{*})^{-1}\mathcal{S}^{\prime *T}+\mathcal{R}^{\prime }, \nonumber \\
&&\Theta_1\varphi =(\mathcal{U}^{*T}(I-\mathcal{A}^{*})^{-1}\mathcal{%
Q}^{\prime }\mathcal{+S}^{\prime })(I-\mathcal{A})^{-1}\varphi +\mathcal{U}%
^{*T}(I-\mathcal{A}^{*})^{-1}\Delta_T^{*}G\varphi (T)+%
\Gamma_T^{*}G\varphi (T), \nonumber \\
&&\Theta_2=(I-\mathcal{A}^{*})^{-1}\mathcal{Q}^{\prime }(I-\mathcal{%
A})^{-1}.
\end{eqnarray}
In above, $A^{T}$ is the transpose of $A$. From (5), (6) and (10) we
have
\begin{eqnarray*}
&&(\mathcal{U}^{*T}(I-\mathcal{A}^{*})^{-1}\mathcal{Q}^{\prime }\mathcal{%
+S}^{\prime })(I-\mathcal{A})^{-1}\mathcal{U} \\
&=&\left(
\begin{array}{c}
\mathcal{B}_1^{*} \\
\mathcal{C}_1^{*}
\end{array}
\right) (I-\mathcal{A}^{*})^{-1}\mathcal{Q}^{\prime }(I-\mathcal{A})^{-1}(%
\mathcal{B}_1,\mathcal{C}_1) \\
&& +\left(
\begin{array}{c}
\mathcal{S}_1 \\
\mathcal{S}_2
\end{array}
\right) (I-\mathcal{A})^{-1}(\mathcal{B}_1,\mathcal{C}_1)+\left(
\begin{array}{c}
\Lambda_T^{*} \\
\Pi_T^{*}
\end{array}
\right) G\Delta_T(I-\mathcal{A})^{-1}(\mathcal{B}_1,\mathcal{C}%
_1),
\end{eqnarray*}
and
\[
\mathcal{U}^{*T}(I-\mathcal{A}^{*})^{-1}\mathcal{S}^{\prime
*T}=\left(
\begin{array}{c}
\mathcal{B}_1^{*} \\
\mathcal{C}_1^{*}
\end{array}
\right) (I-\mathcal{A}^{*})^{-1}\left[ (\mathcal{S}_1^{*},\mathcal{S}_2^{*})+%
\Delta_T^{*}G(\Lambda_T,\Pi%
_T)\right] ,
\]
so we have
\[
\Theta =\left(
\begin{array}{cc}
\Theta_{11} & \Theta_{12} \\
\Theta_{21} & \Theta_{22}
\end{array}
\right) ,
\]
where
\begin{eqnarray}
\Theta_{11}&=&\mathcal{B}_1^{*}(I-\mathcal{A}^{*})^{-1}\mathcal{Q}%
^{\prime }(I-\mathcal{A})^{-1}\mathcal{B}_1+\mathcal{S}_1(I-\mathcal{A})^{-1}%
\mathcal{B}_1+\Lambda_T^{*}G\Delta_T(I-\mathcal{A}%
)^{-1}\mathcal{B}_1 \nonumber \\
&&+\mathcal{B}_1^{*}(I-\mathcal{A}^{*})^{-1}(\mathcal{S}_1^{*}+%
\Delta_T^{*}G\Lambda_T)+\mathcal{R}_{11}+\Lambda%
_T^{*}G\Lambda_T,
\end{eqnarray}
and
\begin{eqnarray}
\Theta_{22}&=&\mathcal{C}_1^{*}(I-\mathcal{A}^{*})^{-1}\mathcal{Q}%
^{\prime }(I-\mathcal{A})^{-1}\mathcal{C}_1+\mathcal{S}_2(I-\mathcal{A})^{-1}%
\mathcal{C}_1+\Pi_T^{*}G\Delta_T(I-\mathcal{A})^{-1}%
\mathcal{C}_1 \nonumber\\
&& +\mathcal{C}_1^{*}(I-\mathcal{A}^{*})^{-1}(\mathcal{S}_2^{*}+%
\Delta_T^{*}G\Pi_T)+\mathcal{R}_{22}+\Pi_T^{*}G%
\Pi_T.
\end{eqnarray}
To conclude this section, we state a necessary and sufficient
condition of existence of saddle point for open-loop game with the
help of Proposition 3.1 and the operators above.
\begin{theorem}
Let (H1) hold, for given $\varphi (\cdot )\in
L_{\mathcal{F}}^2[0,T],$ the
open-loop game admits a saddle point $\widehat{u}\equiv (\widehat{u}_1,%
\widehat{u}_2)$ if and only if $\Theta_{11}\geq 0,$ $%
\Theta_{22}\leq 0$ and $\Theta_1\varphi \in \mathcal{R}(%
\Theta)$, where $\Theta_{11}$ and $\Theta_{22}$ are defined by (12)
and (13). In this case, any saddle point $\widehat{u}$ is a
solution of the following equation: $\Theta u+\Theta%
_1\varphi =0$ with $\Theta_1\varphi$ defined in (11), and it admits
the following representation:
\[
\widehat{u}=-\Theta ^{\dagger }\Theta_1\varphi+(I-\Theta ^{\dagger
}\Theta )v,
\]
for some $v\in L_{\mathcal{F}}^2[0,T]\times L_{\mathcal{F}}^2[0,T].$
In addition, the saddle point is unique if and only if
$\mathcal{N}(\Theta )=\{0\}.$
\end{theorem}
The proof is obvious and we omit it here. Note that here $\Theta_{11}\geq 0,$ $%
\Theta_{22}\leq 0$ is equivalent to the convexity of $u_1\mapsto
J_0(u_1,0)$ and the concavity of $u_2\mapsto J_0(0,u_2)$, where
$J_0(u)$ is the value of $J(u)$ when $\varphi \equiv 0.$
\section{Open-loop games via BSVIE}
In this section, to further characterize explicitly the sufficient
and necessary condition in Theorem 3.1, we will make use of an
efficient tool, i.e., BSVIEs aforementioned. Two equivalent
conditions correspondent to the one in Theorem 3.1 are proposed and
analyzed via BSVIEs. The method is designed around the scheme in
\cite{CY} but with some more delicate and sophisticated analysis
involved. At the outset, we need to prove some lemmas needed in the
sequel.
\begin{lemma}
Let (H1) hold. Then for any $\rho (\cdot )\in
L_{\mathcal{F}}^2[0,T]$, $(\mathcal{A}^{*}\rho )(t)=\sigma (t),$
$t\in [0,T], $ where
\begin{eqnarray}
&&\sigma (t)=E^{\mathcal{F}_{t}}\int_t^T[A_1(s,t)\rho
(s)+A_2(s,t)\nu
(s,t)]ds,  \nonumber \\
&& \rho (t)=E\rho (t)+\int_0^t\nu (t,s)dW(s), \quad t\in[0,T].
\end{eqnarray}
Similarly we have $\forall u_i\in L_{\mathcal{F}}^2[0,T],$ ($i=1,2),$ $(%
\mathcal{B}_1^{*}u_1)(t)=\alpha (t),$
$(\mathcal{C}_1^{*}u_2)(t)=\gamma (t),$ $t\in [0,T],$ where
\begin{eqnarray*}
&&\alpha
(t)=E^{\mathcal{F}_{t}}\int_t^T[B_1(s,t)u_1(s)+B_2(s,t)\beta
(s,t)]ds, \nonumber \\
&&u_1(t)=Eu_1(t)+\int_0^t\beta (t,s)dW(s), \quad t\in [0,T],
\end{eqnarray*}
and
\begin{eqnarray*}
&&\gamma
(t)=E^{\mathcal{F}_{t}}\int_t^T[C_1(s,t)u_2(s)+C_2(s,t)\delta
(s,t)]ds, \nonumber
\\
&&u_2(t)=Eu_2(t)+\int_0^t\delta (t,s)dW(s),\quad  t\in [0,T].
\end{eqnarray*}
\end{lemma}
\begin{proof}Since $\mathcal{A}$ is a bounded linear operator from
the Hilbert space $L^2_{\mathcal{F}}[0,T]$ into itself, thus the
adjoint operator $\mathcal{A}^{*}$ of $\mathcal{A}$ is well-defined.
For any $X(\cdot)\in L^2_{\mathcal{F}}[0,T],$
\begin{eqnarray*}
&&E\int_0^T (\mathcal{A}^{*}\rho)(t)X(t)dt\equiv E\int_0^T
\rho(t)(\mathcal{A}X)(t)dt \nonumber \\
&=& E\int_0^T \rho(t)dt\int_0^tA_1(t,s)X(s)ds+E\int_0^T
\rho(t)dt\int_0^tA_2(t,s)X(s)dW(s) \nonumber \\
&=& E\int_0^T X(t)dt\int_t^TA_1(s,t)\rho(s)ds+ E\int_0^T
X(t)dt\int_t^TA_2(s,t)\nu(s,t)ds \nonumber \\
&=&E\int_0^T X(t)dtE^{\mathcal{F}_{t}}\int_t^T[A_1(s,t)\rho
(s)+A_2(s,t)\nu (s,t)]ds,
\end{eqnarray*}
where we use the relation $\rho (t)=E\rho (t)+\int_0^t\nu
(t,s)dW(s)$ and stochastic Fubini theorem above, thus by the
arbitrariness of $X,$ we get (14). As to the other two results, the
proof is similar.
\end{proof}
\begin{remark}
Let us consider the following equation:
\begin{eqnarray}
Y(t)=\psi
(t)+\int_t^T[A_1(s,t)Y(s)+A_2(s,t)Z(s,t)]ds-\int_t^TZ(t,s)dW(s),
\end{eqnarray}
where $(Y(\cdot ),Z(\cdot ,\cdot ))$ is the unique M-solution of (15) and $%
\psi \in L^2(\Omega \times [0,T])$. From Lemma 4.1, we know that
$Y=E^{\mathcal{F}_{t}}\psi +\mathcal{A}^{*}Y.$ Since
$(I-\mathcal{A})^{-1}$ exists and bounded, we have
$Y=(I-\mathcal{A}^{*})^{-1}E^{\mathcal{F}_{t}}\psi.$
\end{remark}
\begin{lemma}
Let (H1) hold. Then $%
\forall \eta \in L^2(\Omega ),$ we have
\begin{eqnarray}
(\Delta_T^{*}\eta )(s)=A_1(T,s)E^{\mathcal{F}_{s}}\eta
+A_2(T,s)\theta (s), \nonumber
\\
\eta =E\eta +\int_0^T\theta (s)dW(s),\quad t\in[0,T].
\end{eqnarray}
Similarly $\forall \zeta _i\in L^2(\Omega ),$ $(i=1,2),$ we have
\begin{eqnarray*}
&&(\Lambda_T^{*}\zeta _1)(s)=B_1(T,s)E^{\mathcal{F}_{s}}\zeta
_1+B_2(T,s)\kappa _1(s),
\\
&&\zeta _1=E\zeta _1+\int_0^T\kappa _1(s)dW(s),\quad t\in[0,T],
\end{eqnarray*}
and
\begin{eqnarray*}
&&(\Pi_T^{*}\zeta _2)(s)=C_1(T,s)E^{\mathcal{F}_{s}}\zeta _2+C_2(T,s)\kappa _2(s), \\
&&\zeta _2=E\zeta _2+\int_0^T\kappa _2(s)dW(s),\quad t\in[0,T].
\end{eqnarray*}
\end{lemma}
\begin{proof}Because $\Delta$ is a bounded linear operator from
the Hilbert space $L^2_{\mathcal{F}}[0,T]$ into $L^2(\Omega)$, thus
the adjoint operator $\Delta^{*}$ of $\Delta$, which is defined from
$L^2(\Omega)$ into $L^2_{\mathcal{F}}[0,T]$,
is well-defined. For any $\eta \in L^2(\Omega ),$ $X\in L_{\mathcal{F%
}}^2[0,T],$ we have
\begin{eqnarray}
E\int_0^T(\Delta_T^{*}\eta )(s)X(s)ds &=&\left\langle %
\Delta_T^{*}\eta ,X\right\rangle _2=\left\langle \eta ,\Delta%
_TX\right\rangle _1=E\eta \Delta_TX  \nonumber \\
\ &=&E\int_0^TA_1(T,s)\eta X(s)ds+E\int_0^TA_2(T,s)X(s)\eta dW(s) \nonumber \\
\ &=&E\int_0^TA_1(T,s)\eta X(s)ds+E\int_0^TA_2(T,s)\theta (s)X(s)ds  \nonumber\\
\ &=&E\int_0^T[A_1(T,s)\eta +A_2(T,s)\theta (s)]X(s)ds  \nonumber\\
\ &=&E\int_0^T[A_1(T,s)E^{\mathcal{F}_{s}}\eta +A_2(T,s)\theta
(s)]X(s)ds.
\end{eqnarray}
Since $X(\cdot )\in L_{\mathcal{F}}^2[0,T]$ is arbitrary, it follows
from (17) that,
\[
(\Delta_T^{*}\eta )(s)=A_1(T,s)E^{\mathcal{F}_{s}}\eta
+A_2(T,s)\theta (s).
\]
As to the others, the proof is similar.
\end{proof}

The previous two lemmas show the way to express the Hilbert
operators more clearly. The following two theorems are the two main
results in this section, which are established with the help of the
two lemmas above.
\begin{theorem}
Let (H1) hold, then for $i=1,2,$ and any $u_i(\cdot )\in
L_{\mathcal{F}}^2[0,T],$ $(X^{u_1},Y^{u_1},Z^{u_1},\lambda ^{u_1})$
is the unique M-solution of the following decoupled FBSVIE:
\begin{equation}
\left\{
\begin{array}{lc}
X^{u_1}(t)=\displaystyle\int_0^t[A_1(t,s)X^{u_1}(t)+B_1(t,s)u_1(s)]ds \\
\quad \quad \quad +\displaystyle\int_0^t[A_2(t,s)X^{u_1}(t)+B_2(t,s)u_1(s)]dW(s), \\
Y^{u_1}(t)=
Q(t)X^{u_1}(t)+S_1(t)u_1(t)+A_1(T,t)GX^{u_1}(T)+A_2(T,t)\theta
_1(t) \\
\quad \quad \quad +\displaystyle\int_t^T[A_1(s,t)Y^{u_1}(s)+A_2(s,t)Z^{u_1}(s,t)]ds-%
\displaystyle\int_t^TZ^{u_1}(t,s)dW(s),\\
\lambda
^{u_1}(t)=E^{\mathcal{F}_{t}}\displaystyle\int_t^T[B_1(s,t)Y^{u_1}(s)+B_2(s,t)Z
^{u_1}(s,t)]ds,
\end{array}
\right.
\end{equation}
and $(X^{u_2},Y^{u_2},Z^{u_2},\lambda ^{u_2})$ is the unique
M-solution of the following decoupled FBSVIE:
\begin{equation}
\left\{
\begin{array}{lc}
X^{u_2}(t)=\displaystyle\int_0^t[A_1(t,s)X^{u_2}(t)+C_1(t,s)u_2(s)]ds \\
\quad \quad \quad +\displaystyle\int_0^t[A_2(t,s)X^{u_2}(t)+C_2(t,s)u_2(s)]dW(s), \\
Y^{u_2}(t)=
Q(t)X^{u_2}(t)+S_2(t)u_2(t)+A_1(T,t)GX^{u_2}(T)+A_2(T,t)\theta
_2(t) \\
\quad \quad \quad +\displaystyle\int_t^T[A_1(s,t)Y^{u_2}(s)+A_2(s,t)Z^{u_2}(s,t)]ds-%
\displaystyle\int_t^TZ^{u_2}(t,s)dW(s),\\
\lambda
^{u_2}(t)=E^{\mathcal{F}_{t}}\displaystyle\int_t^T[C_1(s,t)Y^{u_1}(s)+C_2(s,t)Z
^{u_2}(s,t)]ds,
\end{array}
\right.
\end{equation}
where $i=1,2,$
\begin{eqnarray*}
GX^{u_i}(T) &=&EGX^{u_i}(T)+\int_0^T\theta _i(s)dW(s).
\end{eqnarray*}
Then $\Theta_{11}\geq 0$ is equivalent to: $\forall u_1(t)\in L_{%
\mathcal{F}}^2[0,T],$
\begin{eqnarray}
&&E\int_0^T[\lambda
^{u_1}(s)+S_1(s)X^{u_1}(s)+R_{11}(s)u_1(s)]u_1(s)ds \nonumber\\
&&+E\int_0^T[B_1(T,s)E^{\mathcal{F}_{s}}GX^{u_1}(T)+B_2(T,s)\theta
_1(s)]u_1(s)ds\geq 0,
\end{eqnarray}
and $\Theta_{22}\leq 0$ is equivalent to: $\forall u_2(t)\in L_{%
\mathcal{F}}^2[0,T],$
\begin{eqnarray}
&&E\int_0^T[\lambda
^{u_2}(s)+S_2(s)X^{u_2}(s)+R_{22}(s)u_2(s)]u_2(s)ds \nonumber\\
&&+E\int_0^T[C_1(T,s)E^{\mathcal{F}_{s}}GX^{u_2}(T)+C_2(T,s)\theta
_2(s)]u_2(s)ds\leq 0.
\end{eqnarray}
\end{theorem}
\begin{proof}It is clear that $\forall u_1(t)\in L_{\mathcal{F}%
}^2[0,T],$
\begin{eqnarray*}
&&\ \mathcal{B}_1^{*}(I-\mathcal{A}^{*})^{-1}(\mathcal{Q}^{\prime }(I-%
\mathcal{A})^{-1}\mathcal{B}_1u_1+\mathcal{S}_1^{*}u_1+\Delta%
_T^{*}G\Lambda_Tu_1) \\
\ &=&\mathcal{B}_1^{*}(I-\mathcal{A}^{*})^{-1}[(\mathcal{Q}+\Delta_T^{*}G\Delta_T)(I-\mathcal{A}%
)^{-1}\mathcal{B}_1u_1+\mathcal{S}_1^{*}u_1
+\Delta_T^{*}G\Lambda_Tu_1] \\
\ &=&\mathcal{B}_1^{*}(I-\mathcal{A}^{*})^{-1}(QX^{u_1}+S_1u_1+%
\Delta_T^{*}GX^{u_1}(T)),
\end{eqnarray*}
and
\begin{eqnarray*}
&& \mathcal{S}_1(I-\mathcal{A})^{-1}\mathcal{B}_1u_1+\mathcal{R}_{11}u_1+%
\Lambda_T^{*}G\Delta_T(I-\mathcal{A})^{-1}\mathcal{B}%
_1u_1+\Lambda_T^{*}G\Lambda_Tu_1 \\
 &=&S_1X^{u_1}+R_{11}u_1+\Lambda_T^{*}GX^{u_1}(T),
\end{eqnarray*}
where $X^{u_{1}}(t)$ and $X^{u_1}(T)$ can be expressed by
\begin{eqnarray*}
X^{u_{1}}(t)&=&\int_0^tB_1(t,s)u_1(s)ds+\int_0^tB_2(t,s)u_1(s)dW(s)\\
&&+\int_0^tA_1(t,s)X^{u_1}(s)ds+\int_0^tA_2(t,s)X^{u_1}(s)dW(s),\\
X^{u_1}(T)&=&\int_0^TB_1(T,s)u_1(s)ds+\int_0^TB_2(T,s)u_1(s)dW(s)\\
&&+\int_0^TA_1(T,s)X^{u_1}(s)ds+\int_0^TA_2(T,s)X^{u_1}(s)dW(s),
\end{eqnarray*}
thereby
\begin{eqnarray*}
(\Theta_{11}u)&=&\mathcal{B}_1^{*}(I-\mathcal{A}%
^{*})^{-1}[QX^{u_1}+S_1u_1+A_1(T,\cdot)E^{\mathcal{F}_{\cdot}}GX^{u_1}(T) \\
&&+A_2(T,\cdot)\theta
_1(\cdot)]+S_1X^{u_1}+R_{11}u_1\\
&&+B_1(T,\cdot)E^{\mathcal{F}_{\cdot}}GX^{u_1}(T)+B_2(T,\cdot)\theta
_1(\cdot),
\end{eqnarray*}
where
\[
GX^{u_1}(T)=EGX^{u_1}(T)+\int_0^T\theta _1(s)dW(s).
\]
By Lemma 4.1 and Lemma 4.2 we have
\begin{eqnarray*}
&&(\Theta_{11}u)(s)=\lambda
^{u_1}(s)+S_1(s)X^{u_1}(s)+R_{11}(s)u_1(s)\\
&&+B_1(T,s)E^{\mathcal{F}_{s}}GX^{u_1}(T)+B_2(T,s)\theta _1(s),
\end{eqnarray*}
where $(Y^{u_1},Z^{u_1},\lambda ^{u_1}$ is the unique M-solution of
the following BSVIEs,
\begin{equation}
\left\{
\begin{array}{lc}
Y^{u_1}(t)=Q(t)X^{u_1}(t)+S_1(t)u_1(t)+A_1(T,t)E^{\mathcal{F}_{t}}GX^{u_1}(T)+A_2(T,t)\theta
_1(t)\\
\quad \quad \quad +\displaystyle \int_t^T[A_1(s,t)Y^{u_1}(s)+A_2(s,t)Z^{u_1}(s,t)]ds-%
\displaystyle\int_t^TZ^{u_1}(t,s)dW(s),\\
\lambda
^{u_1}(t)=E^{\mathcal{F}_{t}}\displaystyle\int_t^T[B_1(s,t)Y^{u_1}(s)+B_2(s,t)Z
^{u_1}(s,t)]ds.
\end{array}
\right.
\end{equation}
In the similar method we have
\begin{eqnarray*}
(\Theta_{22}u) &=&\mathcal{C}_1^{*}(I-\mathcal{A}%
^{*})^{-1}[QX^{u_2}+S_2u_2+\Delta_T^{*}GX^{u_2}(T)] \\
&&\ +S_2X^{u_2}+R_{22}u_2+\Pi
_T^{*}GX^{u_2}(T) \\
\ &=&\lambda
^{u_2}+S_2X^{u_2}+R_{22}u_2+C_1(T,s)E^{\mathcal{F}_{\cdot}}GX^{u_2}(T)+C_2(T,\cdot)\theta
_2(\cdot),
\end{eqnarray*}
and
\[
GX^{u_2}(T)=EGX^{u_2}(T)+\int_0^T\theta _2(s)dW(s),
\]
where $(Y^{u_2},Z^{u_2},\lambda ^{u_2}$ is the unique M-solution of
the following BSVIEs,
\begin{equation*}
\left\{
\begin{array}{lc}
Y^{u_2}(t)
=QX^{u_2}(t)+S_2u_2(t)+A_1(T,t)E^{\mathcal{F}_{t}}GX^{u_2}(T)+A_2(T,t)\theta
_2(t)\\
\quad \quad \quad +\displaystyle\int_t^T[A_1(s,t)Y^{u_2}(s)+A_2(s,t)Z^{u_2}(s,t)]ds-%
\displaystyle\int_t^TZ^{u_2}(t,s)dW(s),\\
\lambda
^{u_2}(t)=E^{\mathcal{F}_{t}}\displaystyle\int_t^T[C_1(s,t)Y^{u_2}(s)+C_2(s,t)Z
^{u_2}(s,t)]ds.
\end{array}
\right.
\end{equation*}
Hence the conclusion hold naturally.
\end{proof}

Note that (18) or (19) admits a unique M-solution
$(X^{u_i},Y^{u_i},Z^{u_i},\lambda ^{u_i})$ by which we mean that
$(Y^{u_i},Z^{u_i})$ is the unique M-solution of the second BSVIE and
$(X^{u_i},\lambda ^{u_i})$ is the unique adapted solution of the
other two equations.
\begin{theorem}
Let (H1) hold, $\varphi (\cdot )\in L_{\mathcal{F}%
}^2[0,T],$ then
\[
(\Theta u)(t)+(\Theta_1\varphi )(t)=\lambda (t)+(SX)(t)%
\mathcal{+}(Ru)(t)+\Xi_1(t)GX(T)+%
\Xi_2(t)\theta (t),
\]
where $\lambda (\cdot )$ satisfies
\begin{eqnarray}
\lambda (t)=E^{\mathcal{F}_{t}}\int_t^T\left[\left(
\begin{array}{c}
B_1(s,t) \\
C_1(s,t)
\end{array}
\right)Y(s)+\left(
\begin{array}{c}
B_2(s,t)  \\
 C_2(s,t)
\end{array}
\right) Z(s,t)\right]ds,
\end{eqnarray}
$Y(\cdot )$ is the unique M-solution of BSVIE
\begin{eqnarray}
Y(t) &=&Q(t)X(t)+S^T(t)u(t)+A_1(T,t)E^{\mathcal{F}_{t}}GX(T)+A_2(T,t)\theta (t)\nonumber \\
&&\ +\int_t^T[A_1(s,t)Y(s)+A_2(s,t)Z(s,t)]ds-\int_t^TZ(t,s)dW(s),
\end{eqnarray}
$GX(T)=EGX(T)+\int_0^T\theta(s)dW(s),$ and $X(t)$ is the unique
solution of (3).
Consequently, the condition $\Theta_1\varphi\in\mathcal{R}(\Theta)$ holds if and only if there is a $\widehat{u}%
(\cdot )$ such that
\begin{eqnarray}
\lambda (t)+(SX)(t)%
\mathcal{+}(R\widehat{u})(t)+\Xi_1(t)GX(T)+%
\Xi_2(t)\theta (t)=0,
\end{eqnarray}
where $\Xi_i$ are defined below.
\end{theorem}
\begin{proof}It follows from (11) that
\begin{eqnarray*}
(\Theta_1\varphi )(t) &=&[\mathcal{U}^{*T}(I-\mathcal{A}^{*})^{-1}%
\mathcal{Q}^{\prime }(I-\mathcal{A})^{-1}\varphi
](t)\mathcal{+[S}^{\prime
}(I-\mathcal{A})^{-1}\varphi ](t) \\
&&\ +[\mathcal{U}^{*T}(I-\mathcal{A}^{*})^{-1}\Delta_T^{*}G\varphi
(T)](t)+[\Gamma_T^{*}G\varphi (T)](t), \\
\  &=&\mathcal{U}^{*T}(I-\mathcal{A}^{*})^{-1}[(QX^\varphi )(t)+(%
\Delta_T^{*}G\Delta_TX^\varphi )(t)+(\Delta%
_T^{*}G\varphi (T))(t)] \\
&&\ +(SX^\varphi )(t)+(\Gamma_T^{*}G\Delta_TX^\varphi )(t)+(\Gamma
_T^{*}G\varphi (T))(t),
\end{eqnarray*}
where
\begin{eqnarray*}
X^{\varphi}(t)=\varphi(t)+\int_0^tA_1(t,s)X^{\varphi}(s)ds+\int_0^tA_2(t,s)X^{\varphi}(s)dW(s).
\end{eqnarray*}
 So we have
\begin{eqnarray*}
&&\ (\Theta u)(t)+(\Theta_1\varphi )(t) \\
\  &=&[\mathcal{U}^{*T}(I-\mathcal{A}^{*})^{-1}(QX+S^Tu)](t) \\
&&\ +[\mathcal{U}^{*T}(I-\mathcal{A}^{*})^{-1}(\Delta_T^{*}G%
\Delta_TX+\Delta_T^{*}G\Gamma_Tu+%
\Delta_T^{*}G\varphi (T))](t) \\
&&\ \mathcal{+}(SX)(t)+(\Gamma_T^{*}G%
\Delta_TX)(t)+(Ru)(t)+(\Gamma_T^{*}G%
\Gamma_Tu)(t)+(\Gamma_T^{*}G\varphi (T))(t) \\
\  &=&[\mathcal{U}^{*T}(I-\mathcal{A}^{*})^{-1}(QX+S^Tu+\Delta%
_T^{*}GX(T))](t) \\
&&\ +(SX)(t)\mathcal{+}(Ru)(t)+\Xi_1(t)E^{\mathcal{F}_{t}}GX(T)+%
\Xi_2(t)\theta (t) \\
\  &=&\lambda (t)+(SX)(t)\mathcal{+}(Ru)(t)+\Xi%
_1(t)E^{\mathcal{F}_{t}}GX(T)+\Xi_2\theta (t),
\end{eqnarray*}
where
\[
\Xi_1(t)=\left(
\begin{array}{c}
B_1(T,t) \\
C_1(T,t)
\end{array}
\right) , \quad \Xi_2(t)=\left(
\begin{array}{c}
B_2(T,t) \\
C_2(T,t)
\end{array}
\right) ,
\]
and $\lambda (\cdot )$ satisfies
\[
\lambda (t)=E^{\mathcal{F}_{t}}\int_t^T\left[\left(
\begin{array}{c}
B_1(s,t) \\
C_1(s,t)
\end{array}
\right)Y(s)+\left(
\begin{array}{c}
B_2(s,t)  \\
 C_2(s,t)
\end{array}
\right)Z(s,t)\right]ds,
\]
where $Y(\cdot )$ is the M-solution of BSVIE
\begin{eqnarray*}
Y(t) &=&Q(t)X(t)+S^T(t)u(t)+A_1(T,t)E^{\mathcal{F}_{t}}GX(T)+A_2(T,t)\theta (t) \\
&&\ +\int_t^T[A_1(s,t)Y(s)+A_2(s,t)Z(s,t)]ds-\int_t^TZ(t,s)dW(s),
\end{eqnarray*}
and $\theta (\cdot )$ is determined by
\[
GX(T)=EGX(T)+\int_0^T\theta (t)dW(t).
\]
Then the conclusion follows.
\end{proof}
\section{Stochastic integral games and coupled FBSIVEs}

\subsection{A sufficient condition for existence of saddle point}

In this subsection, under certain assumptions, a sufficient
condition for the existence of saddle point $\widehat{u}$ will be
given via coupled FBSVIEs. To begin with, we should investigate the
solvability the following coupled FBSVIE on $[0,T],$
\begin{equation}
\left\{
\begin{array}{lc}
X(t)=\varphi (t)+\displaystyle\int_0^t[A_1(t,s)X(s)+B_1(t,s)P(s)]ds \\
\quad \quad \quad + \displaystyle\int_0^t[A_2(t,s)X(s)+B_2(t,s)P(s)]dW(s),\\
Y(t)=\phi _1(t)X(t)+\phi _2(t)P(t)+\displaystyle\int_t^TC_1(s,t)Y(s)ds\\
\quad \quad \quad +\displaystyle\int_t^TC_2(s,t)Z(s,t)ds-\displaystyle\int_t^TZ(t,s)dW(s),\\
P(t)=\displaystyle
E^{\mathcal{F}_{t}}\int_t^T[D_1(s,t)Y(s)+D_2(s,t)Z(s,t)]ds.
\end{array}
\right.
\end{equation}
We will show that it admits a unique M-solution by which we mean
that $(X,Y,Z,P)$ satisfies the FBSVIEs in the usual sense and
moreover the following hold, $Y(t)=EY(t)+\int_0^tZ(t,s)dW(s).$ Note
that the generator in the second equation is independent of $Z(t,s)$
with $(t,s)\in\Delta^c,$ we can transform the above FBSVIE into
anther form
\begin{equation}
\left\{
\begin{array}{lc}
X(t)=\varphi
(t)+\displaystyle\int_0^tA_1(t,s)X(s)ds+\displaystyle\int_0^tA_2(t,s)X(s)dW(s)\\
\quad \quad \quad
+\displaystyle\int_0^tB_1(t,s)\left(E^{\mathcal{F}_{s}}\displaystyle
\int_s^T[D_1(u,s)Y(u)du+D_2(u,s)Z(u,s)]du\right)ds \\
\quad \quad
\quad+\displaystyle\int_0^tB_2(t,s)\left(E^{\mathcal{F}_{s}}\displaystyle
\int_s^T[D_1(u,s)Y(u)du+D_2(u,s)Z(u,s)]du\right)dW(s),\\
Y(t)=\phi
_1(t)X(t)+E^{\mathcal{F}_{t}}\displaystyle\int_t^TC'_1(s,t)Y(s)ds
+E^{\mathcal{F}_{t}}\displaystyle\int_t^TC'_2(s,t)Z(s,t)ds,
\end{array}
\right.
\end{equation}
where $C'_1(s,t)=C_1(s,t)+\phi_2(t)D_1(s,t),$ and
$C'_2(s,t)=C_2(s,t)+\phi_2(t)D_2(s,t).$ Next we turn to study FBSVIE
(27) rather than (26). Basic assumptions imposed on the parameters
in the above equation are summarized as follows,

 (H2) $A_i (B_i) :\Delta\times \Omega \mapsto R$ $(C_i, D_i :\Delta^c\times
\Omega \mapsto R, respectively )$ is $\Bbb{B}(\Delta)\otimes
\mathcal{F}_{T}$-measurable $(\Bbb{B}(\Delta^c)\otimes
\mathcal{F}_{T}-measurable, respectively )$ such that $s\rightarrow
A_i(t,s) (B_i(t,s))$ $(s\rightarrow C_i(s,t) (D_i(s,t)),
respectively)$ is $\mathcal{F}$- progressively measurable for all
$t\in[0,T]$, $(i=1,2),$ $\varphi (\cdot )\in
L_{\mathcal{F}}^2[0,T],$
 $\phi_i$ $(i=1,2)$ is deterministic function. We assume that for any
$t\geq s,$
\begin{eqnarray*}
&&|A_1(t,s)|\leq K_1(t,s), \quad |A_2(t,s)|\leq K_2(t,s), \\
 &&|B_1(t,s)|\leq e^{\beta s}K_3(t,s), \quad |B_2(t,s)|\leq e^{\beta s}K_4(t,s),
 \end{eqnarray*}
 and for any $t\leq s$
 \begin{eqnarray*}
&&|C_1(t,s)|\leq K_5(t,s),\quad |C_2(t,s)|\leq K_6(t,s), \\
&&\phi_2(t)|D_1(t,s)|\leq K_7(t,s),\quad \phi_2(t)|D_2(t,s)|\leq
K_8(t,s),
\end{eqnarray*}
where $q>2$ is a constant, $|\phi _1(t)|\leq \frac 12e^{-\beta t}$
with $\beta>1$ being a constant and
\[
M_1=\sup_{t\in [0,T]}\int_0^tK_1^2(t,s)<\infty , \quad
M_3=\sup_{t\in [0,T]}\int_0^tK_3^2(t,s)<\infty ,
\]
\[
M_2=\sup_{(t,s)\in \Delta }K_2(t,s)<\infty ,\quad M_4=\sup_{(t,s)\in
\Delta }K_4(t,s),
\]
\[
M_5=\sup_{t\in [0,T]}\int_t^TK_5^2(s,t)ds<\infty , \quad
M_6=\sup_{t\in [0,T]}\int_t^TK_6^q(s,t)ds<\infty ,
\]
\[
M_7=\sup_{t\in [0,T]}\int_t^TK_7^2(s,t)ds<\infty , \quad
M_8=\sup_{t\in [0,T]}\int_t^TK_8^q(s,t)ds<\infty .
\]
\begin{theorem}
Let (H2) hold, then FBSVIE (27) admits a unique M-solution.
\end{theorem}
\begin{proof}Let $\mathcal{M}^2[0,T]$ be the set of element $(Y,Z)$ in $\mathcal{H}^2[0,T]$ such
that $$Y(t)=EY(t)+\int_0^tZ(t,s)dW(s),\quad (t,s)\in\Delta.$$ It is
easy to see that $\mathcal{M}^2[0,T]$ is a closed subspace of
$\mathcal{H}^2[0,T],$ see \cite{Y2}. We consider the following
equation
\begin{equation}
\left\{
\begin{array}{lc}
X(t)=\varphi
(t)+\displaystyle\int_0^tA_1(t,s)x(s)ds+\displaystyle\int_0^tA_2(t,s)x(s)dW(s)\\
\quad \quad
+\displaystyle\int_0^tB_1(t,s)\left(E^{\mathcal{F}_{s}}\displaystyle
\int_s^T[D_1(u,s)y(u)du+D_2(u,s)z(u,s)]du\right)ds \\
\quad
\quad+\displaystyle\int_0^tB_2(t,s)\left(E^{\mathcal{F}_{s}}\displaystyle
\int_s^T[D_1(u,s)y(u)du+D_2(u,s)z(u,s)]du\right)dW(s),\\
Y(t)=\phi
_1(t)x(t)+E^{\mathcal{F}_{t}}\displaystyle\int_t^TC'_1(s,t)y(s)ds
+E^{\mathcal{F}_{t}}\displaystyle\int_t^TC'_2(s,t)z(s,t)ds,
\end{array}
\right.
\end{equation}
for any $\varphi (\cdot )\in L_{\mathcal{F}}^2[0,T],$ and $(x(\cdot
),y(\cdot ),z(\cdot ,\cdot ))\in L_{\mathcal{F}%
}^2[0,T]\times \mathcal{M}^2[0,T]%
.$ Obviously (28) admits a unique adapted M-solution $(X(\cdot
),Y(\cdot
),Z(\cdot ,\cdot ))\in L_{\mathcal{F}%
}^2[0,T]\times \mathcal{M}^2[0,T],$ and we can define a
map $\Theta :L_{\mathcal{F}}^2[0,T]\times \mathcal{M}^2[0,T]%
\rightarrow L_{\mathcal{F}}^2[0,T]\times \mathcal{M}^2[0,T]$ by
\begin{eqnarray*}
&&\Theta (x(\cdot ),y(\cdot ),z(\cdot ,\cdot )) =(X(\cdot ),Y(\cdot
),Z(\cdot ,\cdot )),
\nonumber \\
&&\forall (x(\cdot ),y(\cdot ),z(\cdot ,\cdot )) \in
L_{\mathcal{F}}^2[0,T]\times \mathcal{M}^2[0,T].
\end{eqnarray*}
Let $(\overline{x}(\cdot ),\overline{y}(\cdot ),\overline{z}(\cdot ,\cdot )%
)\in L_{\mathcal{F}%
}^2[0,T]\times \mathcal{M}^2[0,T]$ and $$\Theta (%
\overline{x}(\cdot ),\overline{y}(\cdot ),\overline{z}(\cdot ,\cdot )%
)=(\overline{X}(\cdot ),%
\overline{Y}(\cdot ),\overline{Z}(\cdot ,\cdot )%
).$$ As to the first forward equation in (28),
\begin{eqnarray}
&&E\int_0^Te^{-\beta t}|X(t)-\overline{X}(t)|^2dt \nonumber \\
&\leq &2E\int_0^Te^{-\beta t}\left| \int_0^tA_1(t,s)[x(s)-\overline{x}%
(s)]+B_1(t,s)[p(s)-\overline{p}(s)]ds\right| ^2dt \nonumber \\
&&+4E\int_0^Te^{-\beta t}\left( \int_0^tA_2^2(t,s)[x(s)-\overline{x}%
(s)]^2ds+\int_0^tB_2^2(t,s)[p(s)-\overline{p}(s)]^2ds\right) dt \nonumber \\
&\leq &4(M_1+M_2)E\int_0^T|x(s)-\overline{x}(s)|^2\int_s^Te^{-\beta t}dt \nonumber\\
&&+4(M_3+M_4)E\int_0^T|p(s)-\overline{p}(s)|^2e^{2\beta
s}\int_s^Te^{-\beta
t}dt \nonumber \\
&\leq &\frac C\beta E\int_0^T|x(s)-\overline{x}(s)|^2e^{-\beta
s}ds+\frac C\beta E\int_0^T|p(s)-\overline{p}(s)|^2e^{\beta s}ds,
\nonumber
\end{eqnarray}
where we denote
$$p(s)-\overline{p}(s)=E^{\mathcal{F}_{s}}\displaystyle
\int_s^TD_1(u,s)[y(u)-\overline{y}(u)]du+E^{\mathcal{F}_{s}}\displaystyle
\int_s^TD_2(u,s)[z(u,s)-\overline{z}(u,s)]du.$$ Obviously we have
\begin{eqnarray}
&&E\int_0^Te^{\beta t}|p(t)-\overline{p}(t)|^2dt
 \nonumber\\
&\leq &2M_7E\int_0^Te^{\beta t}\int_t^T|y(s)-\overline{y}(s)|^2dsdt+2M_8E%
\int_0^Te^{\beta t}\int_t^T|z(s,t)-\overline{z}(s,t)|^2dsdt \nonumber \\
&\leq &\frac{2M_7}\beta E\int_0^Te^{\beta
s}|y(s)-\overline{y}(s)|^2ds+2M_8E\int_0^Te^{\beta
s}|y(s)-\overline{y}(s)|^2ds, \nonumber
\end{eqnarray}
consequently
\begin{eqnarray}
&&E\int_0^Te^{-\beta t}|X(t)-\overline{X}(t)|^2dt \nonumber\\
&\leq& \frac C\beta E\int_0^T|x(s)-\overline{x}(s)|^2e^{-\beta
s}ds+\frac C\beta E\int_0^T|y(s)-\overline{y}(s)|^2e^{\beta s}ds,
\end{eqnarray}
where $C$ depends on $M_i$, $i=1,2,3,4,7,8.$ As to the other one in
(28), for some $p\in(1,2),$ and $\frac{1}p+\frac {1}q=1,$
\begin{eqnarray}
&&E\int_0^Te^{\beta t}|Y(t)-\overline{Y}(t)|^2dt \nonumber \\
&\leq &2E\int_0^Te^{\beta t}\phi _1^2(t)|x(t)-\overline{x}%
(t)|^2dt+
C(M_5+M_7)E\int_0^Te^{\beta t}\int_t^T|y(s)-\overline{y}(s)|^2dsdt \nonumber \\
&&+C(M_6+M_8)E%
\int_0^Te^{\beta t}\left(\int_t^T|z(s,t)-\overline{z}(s,t)|^pds\right)^{\frac 2 p}dt  \nonumber\\
&\leq &\frac 12E\int_0^Te^{-\beta t}|x(t)-\overline{x}(t)|^2dt
+\frac{C}\beta E\int_0^Te^{\beta s}|y(s)-\overline{y}(s)|^2ds\nonumber \\
&&+ C\left[\frac1 \beta\right]^{\frac {2-p} {p}}
E\int_0^Tdt\int_t^Te^{\beta s}|z(s,t)-\overline{z}(s,t)|^2ds \nonumber \\
&\leq & C\left(\left[\frac1 \beta\right]^{\frac {2-p} {p}}+\frac {1
}\beta \right)E\int_0^Te^{\beta s}|y(s)-\overline{y}(s)|^2ds+\frac
12E\int_0^Te^{-\beta t}|x(t)-\overline{x}(t)|^2dt,
\end{eqnarray}
where $C$ depends on $M_i$, $i=5,6,7,8.$ Note that here we use the
following fact, for $1<p<2,$ and $r>0,$
\begin{eqnarray}
&&\left[ \int_t^T|z(s,t)-\overline{z}(s,t)|^{p}ds\right] ^{\frac 2
p} \nonumber \\
&\leq& \left[ \int_t^Te^{-rs\frac 2{2-p}}ds\right] ^{\frac{2-p}%
{p}}\int_t^Te^{rs\frac 2 p}|z(s,t)-\overline{z}(s,t)|^2ds \nonumber \\
&\leq &\left[ \frac 1 r\right] ^{\frac{2-p}{p}}\left[ \frac{2-p}2\right] ^{%
\frac{2-p}{p}}e^{-rt\frac 2 p}\int_t^Te^{rs\frac 2
p}|z(s,t)-\overline{z}(s,t)|^2ds.\nonumber
\end{eqnarray}
By (29) and (30) we obtain
\begin{eqnarray*}
&&E\int_0^Te^{-\beta t}|X(t)-\overline{X}(t)|^2dt+E\int_0^Te^{\beta t}|Y(t)-%
\overline{Y}(t)|^2dt \\
&\leq &\frac C\beta E\int_0^T|x(s)-\overline{x}(s)|^2e^{-\beta
s}ds+\frac 12E\int_0^Te^{-\beta t}|x(t)-\overline{x}(t)|^2dt\\
&&+C\left(\left[\frac1 \beta\right]^{\frac {2-p} {p}}+\frac {1
}\beta \right)E\int_0^Te^{\beta s}|y(s)-\overline{y}(s)|^2ds,
\end{eqnarray*}
where $C$ depends on $M_i(i=1,2\cdots 8).$ So we can choose a
suitable $\beta $ such that the mapping $\Theta$ is contracted and
the result holds naturally.
\end{proof}

Suppose $G=0,$ $R_{11}>0$ and $R_{22}<0,$ then $R^{-1}$ exists which
is expressed by
$$R^{-1}(t)=\left[\begin{array}{cc}
 A^{-1}(t)&  -A^{-1}(t)R_{12}(t)R^{-1}_{22}(t)\\
-B^{-1}(t)R_{21}(t)R^{-1}_{11}(t) & B^{-1}(t)
\end{array}
\right], $$ where $A(t)=R_{11}(t)-R_{12}(t)R_{22}(t)R_{21}(t),\quad
B(t)=R_{22}(t)-R_{21}(t)R^{-1}_{11}(t)R_{12}(t).$ Moreover we assume
$A^{-1}(t)$, $B^{-1}(t)$, $R_{11}(t)$ and $R_{22}(t)$ are bounded,
thus $R(t)^{-1}$ is uniformly bounded, then (23) can be rewritten as
$u(t)=-R(t)^{-1}[S(t)X(t)+\lambda(t)]$ with $t\in[0,T].$ After
substituting $u(t)$ into (3), (23) and (24), we obtain the
following:
\begin{equation}
\left\{
\begin{array}{lc}
X(t)=\varphi (t)+\displaystyle\int_0^t[(A_1(t,s)-U_1(t,s)R^{-1}(s)S(s))X(s)-U_1(t,s)R^{-1}(s)\lambda(s)]ds \\
\quad \quad \quad + \displaystyle\int_0^t[(A_2(t,s)-U_2(t,s)R^{-1}(s)S(s))X(s)-U_2(t,s)R^{-1}(s)\lambda(s)]dW(s),\\
Y(t)=[Q(t)-S^{T}(t)R^{-1}(t)S(t)]X(t)-S^{T}(t)R^{-1}(t)\lambda(t)+\displaystyle\int_t^T A_1(s,t)Y(s)ds\\
\quad \quad \quad +\displaystyle\int_t^TA_2(s,t)Z(s,t)ds-\displaystyle\int_t^TZ(t,s)dW(s),\\
\lambda
(t)=E^{\mathcal{F}_{t}}\displaystyle\int_t^T\left[U_1^T(s,t)Y(s)+U_2^T(s,t)Z(s,t)\right]ds,
\end{array}
\right.
\end{equation}
where $U_i(t,s)=\left( B_i(t,s), C_i(t,s) \right)$ $(i=1,2).$

The preceding theorem implies that if $[Q(t)-S^{T}(t)R^{-1}(t)S(t)]$
satisfies certain condition, then (31) admits a unique M-solution
$(X,Y,Z,\lambda),$ thereby the following result is straightforward.
\begin{theorem}
Let (H1) hold, $[Q(t)-S^{T}(t)R^{-1}(t)S(t)]<\frac 1 2 e^{-\beta
t}$, where $\beta$ is a constant depending on the upper boundary of
the coefficients in the game problem, moreover, $R^{-1}(t)$ is
bounded, then (31) admits a unique M-solution $(X,Y,Z,\lambda).$
Furthermore, if (20) and (21) hold, then the quadratic integral game
admits an open-loop saddle point $\widehat{u},$ and it admits a
representation, $u(t)=-R(t)^{-1}[S(t)X(t)+\lambda(t)]$.
\end{theorem}
\subsection{Some furthermore considerations on stochastic integral games}
In this subsection, we would like to give some furthermore
considerations along the routine above. As to the case of $G\neq 0,$
if we define $u(t)$ as $u(t)=-R^{-1}(t)\lambda(t)$, then (3), (23)
and (24) can be rewritten as
\begin{equation}
\left\{ \begin{array}{lc}
 X(t)=\varphi
(t)-\displaystyle\int_0^tU_1(t,s)R^{-1}(s)\lambda(s)ds
-\displaystyle\int_0^tU_2(t,s)R^{-1}(s)\lambda(s)dW(s)\\
\quad \quad \quad
+\displaystyle\int_0^tA_1(t,s)X(s)ds+\displaystyle\int_0^tA_2(t,s)X(s)dW(s),\\
Y(t)=Q(t)X(t)-S(t)^TR^{-1}(t)\lambda(t)+A_1^{T}(T,t)GX(T)+\displaystyle\int_t^T
A_1(s,t)Y(s)ds\\
\quad \quad \quad +A_2^{T}(T,t)\theta(t)+\displaystyle\int_t^TA_2(s,t)Z(s,t)ds-\displaystyle\int_t^TZ(t,s)dW(s),\\
\lambda
(t)=E^{\mathcal{F}_{t}}[U_1^{T}(T,t)GX(T)+U_2^{T}(T,t)\theta(t)]
+S(t)X(t)\\
\quad \quad \quad
+E^{\mathcal{F}_{t}}\displaystyle\int_t^T\left[U_1^T(s,t)Y(s)+U_2^T(s,t)Z(s,t)\right]ds,
\end{array}
\right.
\end{equation}
where $GX(T)=EGX(T)+\int_0^T\theta(s)dW(s).$ Obviously (32) is
coupled FBSVIE. In some special case, for example, $S(t)=0,$ $R$ is
uniform positive, $Q$ and $G$ are non-negative, then the above
FBSVIE (32) admits a unique M-solutions, see p.75 in \cite{Y2}.
However, as to the general case, the solvability problem is still a
question for us to endeavor to overcome. One main technical obstacle
is how to deal with the appearance of $GX(T)$ in the second
equation, nonetheless, it is just the reason, we believe, that the
problem has certain relations with the solvability for some
stochastic Fredholm-Volterra
 integral equation. To get some feeling about this, Let us consider
 a special case below.
 We assume that all the coefficients aforementioned are deterministic,
  $A_i(t,s)=0,$ $(i=1,2)$, $\varphi (\cdot)=\varphi_1(t)+\int_0^tl(t,s)dW(s),$
  $\varphi_1$ and $l$ are deterministic functions.
In such special setting, $Y(t)=Q(t)X(t)-S^T(t)R^{-1}(t)\lambda(t),$
$t\in[0,T],$ and $Z(t,s)=0,$ $0\leq t\leq s\leq T.$
Due to the martingale representation theorem, there must exists a
unique process $\pi,$ such that
\begin{eqnarray}
\lambda(t)=E\lambda(t)+\int_0^t \pi(t,s)dW(s),\quad t\in[0,T],\\
X(t)=EX(t)+\int_0^t K(t,s)dW(s),\quad t\in[0,T],
\end{eqnarray}
thus we can express $Z(t,s),$ $(t,s)\in \Delta$ by
\begin{eqnarray*}
Z(t,s)=Q(t)K(t,s)-S^T(t)R^{-1}(t)\pi(t,s).
\end{eqnarray*}
On the other hand,
\begin{eqnarray}
&&GX(T)=G\varphi
(T)-\int_0^TGU_1(T,s)R^{-1}(s)\lambda(s)ds
-\int_0^TGU_2(T,s)R^{-1}(s)\lambda(s)dW(s),\nonumber \\
\end{eqnarray}
 then substitute (35) into $GX(T)=EGX(T)+\int_0^T\theta(s)dW(s)$ and by stochastic Fubini theorem, we have
$$\theta(s)=-\int_s^TGU_1(T,u)R^{-1}(u)\pi(u,s)du-GU_2(T,s)R^{-1}(s)\lambda(s)+l(T,s).$$
Similarly we get
$$K(t,s)=-\int_s^tU_1(t,u)R^{-1}(u)\pi(u,s)du-U_2(t,s)R^{-1}(s)\lambda(s)+l(t,s).$$
Then put the expression of $X$ and $Y$ into the third equation of
(32) we have
\begin{eqnarray}
\lambda (t)&=&\Sigma_1(t)+E^{\mathcal{F}_{t}}\int_0^T\Sigma_2(t,s)\lambda(s)ds+E^{\mathcal{F}_{t}}\int_0^T\Sigma_3(t,s)\lambda(s)dW(s)\nonumber\\
&&+E^{\mathcal{F}_{t}}\int_t^T\Sigma_4(t,s)\pi(s,t)ds+E^{\mathcal{F}_{t}}\int_t^T\Sigma_5(t,s)\lambda(s)ds \nonumber\\
&&+E^{\mathcal{F}_{t}}\int_t^T\Sigma_6(t,s)\lambda(s)dW(s)+\Sigma_7(t)\lambda(t) \nonumber\\
&&+E^{\mathcal{F}_{t}}\int_0^t\Sigma_8(t,s)\lambda(s)dW(s)+E^{\mathcal{F}_{t}}\int_0^t
\Sigma_9(t,s)\lambda(s)ds \nonumber\\
&=&\Sigma_1(t)+E^{\mathcal{F}_{t}}\int_0^T\Sigma_2'(t,s)\lambda(s)ds+E^{\mathcal{F}_{t}}\int_0^T\Sigma_3'(t,s)\lambda(s)dW(s)\nonumber\\
&&+E^{\mathcal{F}_{t}}\int_t^T\Sigma_4(t,s)\pi(s,t)ds+E^{\mathcal{F}_{t}}\int_t^T\Sigma_5'(t,s)\lambda(s)ds \nonumber\\
&&+E^{\mathcal{F}_{t}}\int_t^T\Sigma_6'(t,s)\lambda(s)dW(s)+\Sigma_7(t)\lambda(t),
\end{eqnarray}
where
\begin{eqnarray*}
&&\Sigma_1(t)=U_1^{T}(T,t)G\varphi(T)+\int_t^TU_1^{T}(s,t)Q(s)\varphi(s)ds+S(t)\varphi(t)\\
&&+E^{\mathcal{F}_{t}}\int_t^T\left[U_1^T(s,t)Q(s)\int_0^tl(s,u)dW(u)ds+U^T_2(s,t)Q(s)l(s,t)\right]ds\\
 &&\Sigma_2(t,s)=-GU_1^{T}(T,t)U_1(T,s)R^{-1}(s);
\Sigma_3(t,s)=-GU_1^{T}(T,t)U_2(T,s)R^{-1}(s);\\
&&\Sigma_4(t,s)=-U_2^{T}(T,t)GU_1(T,s)R^{-1}(s)-U_2^T(s,t)S^T(s)R^{-1}(s)\\
&&-\int_s^TU_2^{T}(u,t)Q(u)U_1(u,s)duR^{-1}(s);\\
&&\Sigma_5(t,s)=-\int_s^TU_1^T(u,t)Q(u)U_1(u,s)duR^{-1}(s)-U_1^T(s,t)S(s)R^{-1}(s);\\
&&\Sigma_6(t,s)=-\int_s^TU_1^T(u,t)Q(u)U_2(u,s)duR^{-1}(s);\\
&&\Sigma_7(t,s)=-U_2^T(T,t)GU_2(T,t)R^{-1}(t)-\int_t^TU_2^T(s,t)Q(s)U_2(s,t)dsR^{-1}(t);\\
&&\Sigma_8(t,s)=-\int_t^TU_1^T(u,t)Q(u)U_2(u,s)duR^{-1}(s)-S(t)U_2(t,s)R^{-1}(s);\\
&&\Sigma_9(t,s)=-\int_t^TU_1^T(u,t)Q(u)U_1(u,s)duR^{-1}(s)-S(t)U_1(t,s)R^{-1}(s);\\
&&\Sigma_2'(t,s)=\Sigma_2(t,s)+\Sigma_9(t,s);\quad
\Sigma_3'(t,s)=\Sigma_3(t,s)+\Sigma_8(t,s);\\
&&\Sigma_5'(t,s)=\Sigma_5(t,s)-\Sigma_9(t,s);\quad
\Sigma_6'(t,s)=\Sigma_6(t,s)-\Sigma_8(t,s).
\end{eqnarray*}
We can denote (36) as a linear BSFVIE. To sum up,
\begin{theorem}
Let all the coefficients are deterministic, $A_i(t,s)=0,$ $i=1,2$,
$\varphi(t)=\varphi_1(t)+\int_0^tl(t,s)dW(s),$ $\varphi_1$ and $l$
are deterministic functions, $R^{-1}(t)$ exists and bounded. If (36)
admits a solution $\lambda,$ furthermore, we assume (20) and (21)
hold, then the quadratic integral game admits an open-loop saddle
point $\widehat{u},$ and it admits a representation
$\widehat{u}(t)=-R^{-1}(t)\lambda(t)$.
\end{theorem}
More specially, suppose that $U_2=0,$ then the preceding BSFVIE (36)
becomes
\begin{eqnarray}
\lambda (t)
&=&\Sigma_1(t)+E^{\mathcal{F}_{t}}\int_0^T\Sigma_2''(t,s)\lambda(s)ds+E^{\mathcal{F}_{t}}\int_0^t
\Sigma_9''(t,s)\lambda(s)ds,
\end{eqnarray}
where $\Sigma_2''(t,s)=\Sigma_2(t,s)+\Sigma_5(t,s)$ and
$\Sigma_9''(t,s)=\Sigma_9(t,s)-\Sigma_5(t,s).$ Equation (37) is a
forward stochastic Fredholm-Volterra integral equation (SFVIE for
short), thereby under some assumptions the above SFVIE admits a
unique solution $\lambda,$ see \cite{CW}, \cite{S} and the reference
cited therein. Next we will present a example to show the
application of the above results.

Let us consider a stochastic delay equation of the form
\begin{eqnarray}
dX(t)&=&\left[A_1'(t)X(t)+A_2'(t)X(t-h)+\int_{t-h}^tA_0'(t,s)X(s)ds+B_1'(t)u_1(t)
\right. \nonumber \\
&&\left.+B_2'(t)u_1(t-h)+C_1'(t)u_2(t)+C_2'(t)u_2(t-h)\right]+D'(t)dW(t),
\end{eqnarray}
with $t\in[0,T]$ where $X(t)=k(t)$ with $t\in[-h,0],$ $A_j'$,
$B_i'$, $C_i'$, $D'$ and $k$ are bounded deterministic functions,
$(i=1,2,j=0,1,2),$ $B_2'\equiv0,$ $C_2'\equiv0$ for $t<h$, the delay
$h>0.$ Notice that when $A_0'\equiv 0,$ $D'\equiv0$ and we consider
the system in a deterministic setting, then (38) will degenerate
into the one in Section 7 of \cite{Y3}. It was shown in \cite{L2},
see also \cite{OZ}, that this type of delay equation can be written
in the following equivalent form:
\begin{eqnarray}
X(t)=X_0(t)+\int_0^t[K_1(t,s)u_1(s)+K_2(t,s)u_2(s)]ds+\int_0^t\Phi(t,s)D'(s)dW(s)
\end{eqnarray}
where
$$X_0(t)=\Phi(t,0)k(0)+\int_{-h}^0\left[\Phi(t,s+h)A_2'(s+h)+\int_0^h\Phi(t,u)A_0'(u,s)du\right]k(s)ds
,$$ $K_1(t,s)=\Phi(t,s)B_1'(s)+\Phi(t,s+h)B_2'(s+h),$
$K_2(t,s)=\Phi(t,s)C_1'(s)+\Phi(t,s+h)C_2'(s+h),$ and $\Phi$ is the
transition function:
\begin{eqnarray*}
\frac{\partial\Phi}{\partial
t}(t,s)=A_1'(t)\Phi(t,s)+A_2'(t)\Phi(t-h,s)+\int_{t-h}^tA_0'(t,u)\Phi(u,s)du
\end{eqnarray*}
with $t\in[0,T]$, $\Phi(s,s)=1$ and $\Phi(t,s)=0$ with $t<0.$
Obviously (39) is a simple form of the forward equation in (32). In
this case, $\varphi_1(t)=X_0(t)$, $l(t,s)=\Phi(t,s)D'(s),$
$U_1=(K_1, K_2),$ then we get
\begin{theorem}
Let the dynamic system is described by a stochastic delay equation
(38), and the cost functional is defined by (4), $R^{-1}$ is
bounded. If the SFVIE (37) admits a solution, furthermore, (20) and
(21) hold, then the quadratic integral game admits an open-loop
saddle point $\widehat{u},$ and it admits a representation
$\widehat{u}(t)=-R^{-1}(t)\lambda(t)$.
\end{theorem}

Furthermore, by assuming $S(t)=0,$ $R_{12}=R_{21}=0,$ we can obtain
one express for the saddle point by
\begin{eqnarray*}
u_1(t)=-R^{-1}_{11}(t)\left[K_1(T,t)E^{\mathcal{F}_{t}}GX(T)+E^{\mathcal{F}_{t}}\int_t^TK_1(s,t)Q(s)X(s)ds\right],
\quad t\in[0,T],\\
u_2(t)=-R^{-1}_{22}(t)\left[K_2(T,t)E^{\mathcal{F}_{t}}GX(T)+E^{\mathcal{F}_{t}}\int_t^TK_2(s,t)Q(s)X(s)ds\right],
\quad t\in[0,T].
\end{eqnarray*}

\end{document}